\documentclass[11pt]{amsart}
\usepackage{amsopn}
\usepackage{amssymb}
\usepackage{graphicx}
\usepackage[all]{xy}
\usepackage{amscd}
\usepackage{color}
\usepackage[colorlinks]{hyperref}

 \newtheorem{thm}{Theorem}[section]
 \newtheorem{cor}[thm]{Corollary}
 \newtheorem{lem}[thm]{Lemma}
 
 \newtheorem{rem}[thm]{Remark}
 \newtheorem{defn}[thm]{Definition}

 \theoremstyle{definition}
 \theoremstyle{remark}

 \numberwithin{equation}{section}
 \newcommand{\p}{\frak p}

\begin{document}

\title[On the equivalence of $FSF$ and weakly Laskerian classes]
{On the equivalence of \\${\rm FSF}$ and weakly Laskerian classes}

\author[K. bahmanpour and A. Khojali  ]{K. bahmanpour and A. Khojali$^*$  }
\address{\newline Department of Mathematics,\newline University of Mohaghegh Ardabili,\newline Ardabil,\newline
Iran.\newline E-mail Address: khojali@uma.ac.ir, khojali@gmail.com}
\address{\newline Department of Mathematics,\newline Islamic Azad University-Ardabil branch,\newline PO Box 5614633167,\newline Ardabil,
\newline Iran\newline E-mail addresses: bahmanpour.k@gmail.com, bahmanpour30@yahoo.com}
\email{}

\thanks{$^*$Corresponding author}

\thanks{}

\subjclass[2000]{}


\date{}

\dedicatory{}

\commby{}


\begin{abstract}
It is proved that, over a Noetherian ring $R$, the class of weakly
Laskerian and ${\rm FSF}$ modules are the same classes. By using
this characterization we proved that the property of being weakly
Laskerian descends by finite integral extensions of local ring
homomorphisms and ascends by tensoring under the completion.
\end{abstract}

\maketitle
\section*{\textbf{Introduction}}
\smallskip
Throughout this paper, $R$ will denote a commutative ring with
identity element $1_R$ and all modules are considered unitary. It is
well known that if $R$ is a Noetherian ring and $M$ a finitely
generated $R$-module, then for all proper submodule $N$ of $M$ the
set of associated prime ideals of $M/N$, abbreviated as ${\rm
Ass}_R(M/N)$ or ${\rm Ass}(M/N)$ if there is no confusion about the
underlying ring $R$, is a finite set. In fact, the class of
Laskerian modules has the property that the set of associated prime
ideals of any quotient of a Laskerian module is a finite set and it
is not hard to find examples of Laskerian modules which is not
finitely generated. Recall that an $R$-module is called
\emph{Laskerian} if every proper submodule is an intersection of a
finite number of primary submodules. The importance of the
finiteness of ${\rm Ass}(M/N)$, for a submodule $N$ of an $R$-module
$M$, comes back to the fact that the finiteness of ${\rm Ass}(M/N)$
guarantees the existence of regular elements on $M/N$, which is of
the most importance in the study of cohomological properties of
$M/N$. Thinking of these facts and toward a general study of Local
cohomology modules  in [\ref{divmafi1}], Divaani-Aazar and Mafi
introduced the class of weakly Laskerian modules and proved some
interesting properties of their local cohomology modules which were
known just for finite modules. Recall that an $R$-module $M$ is
called \emph{weakly Laskerian} if ${\rm Ass}(M/N)$ is a finite set
for each proper submodule $N$ of $M$. It is easy to see that the
class of Artinian, Laskerian and Minimax modules are contained in
the class of weakly Laskerian modules. Recently, Hung Quy
[\ref{hung}], introduced the class of ${\rm FSF}$ modules, modules
containing some finitely generated submodules such that the support
of the quotient module is finite, and proved the following result.
\begin{thm}
Assume that $\frak a$ is an ideal of the Noetherian ring R, and let
$M$ be a finitely generated R-module. Let $t\in \Bbb{N}_0$ be such
that either $H^i_{\frak a}(M)$ is finite or ${\rm Supp}(H^i_{\frak
a} (M))$ is finite for all $i<t$. Then ${\rm Ass}(H^t_{\frak a}(M))$
is finite.
\end{thm}
See also [\ref{Brodmannfaghani}, Proposition 2.1].  The main purpose
of this paper is to prove that over a Noetherian ring $R$, for an
$R$-module the property of being weakly Laskerian or ${\rm FSF}$ are
the same. By using this characterization we prove that the property
of being weakly Laskerian ascends by tensoring under the completion
and descends by finite integral extensions of local ring
homomorphisms. More precisely, it is proved that:\\
\begin{thm}
Let $R$ be a Noetherian ring and $M$ a nonzero $R$-module. The
following statements are equivalent:
\begin{enumerate}
  \item $M$ is a weakly Laskerian module;
  \item M is an $\emph{{\rm FSF}}$ module.
\end{enumerate}
\end{thm}\label{2}
\begin{thm}
Let $(R, \frak m)$ be a Noetherian local ring and $M$ a weakly
Laskerian $R$-module. Then $M\otimes_RR^{*}$ is a weakly Laskerian
$R^{*}$-module, where $R^{*}$ denotes the $\frak m$-adic completion
of $R$.
\end{thm}
\begin{thm}
Let $(R, \frak m) \rightarrow (S, \frak n)$ be a Noetherian local
ring homomorphism in which $S$ is a finite $R$-module. Then every
weakly
Laskerian $S$-module is a weakly Laskerian $R$-module.\\
\end{thm}


\section{\textbf{Preliminaries}}
\vspace{0.5cm}In this section we bring some technical results, which
will be used later.\\

The following Lemma is a well known fact [\ref{brodmannsharp},
Proposition 2.1.4], but the proof which is given here is elementary
and included for completeness.\\

\begin{defn}
Let $R$ be ring and $I$ an ideal of $R$. For an $R$-module $M$ the
$I$-torsion submodule of $M$ is denoted by $\Gamma_{_{I}}(M)$ and
defines as
\[\Gamma_{_{I}}(M)=\{x\in M\, |\, \exists n \in
\Bbb{N}\,\,\mbox{such that}\,\, I^nx=0\}.
\]
\end{defn}\label{l1}

\begin{lem} Let $R$ be a Noetherian ring, and $I$ an ideal of $R$.
If $E$ is an injective $R$-module, then so is $\Gamma_{_{I}}(E)$ an
injective $R$-module.
\end{lem}
\begin{proof}
Let $E_1$ be the injective envelope of $\Gamma_{_{I}}(E)$. It is
enough to prove that $E_1=\Gamma_{_{I}}(E)$. By the way of
contradiction, assume that $E_1$ properly contains
$\Gamma_{_{I}}(E)$ and let
\[
\mathcal{S}=\{(0:_Re)\,|\,e\in E_1\setminus \Gamma_{_{I}}(E)\}.
\]
$R$ being a Noetherian ring implies that $\mathcal{S}$ has, at
least, a maximal element, $\p=(0:_Rx)$ say. We claim that $\p$ is a
prime ideal. Assume the contrary. Then, for some $a, b \in R$,  we
have $ab\in \p$, while neither $a$ nor $b$ does not belong to $\p$.
Since $\p=0:_Rx$ is a proper subset of $(0:_Rbx)$ and $\p$ is a
maximal element of $\mathcal{S}$, then $bx\in \Gamma_{_{I}}(E)$.
Therefore, there exists a natural integer $t$ such that $I^tbx=0$.
If $I^t=(c_1,\dots,c_n)$, then $bc_kx=0\,\,\mbox{for all}\,\,
k=1,\dots, n$. Again, as discussed above, by maximality of $\p$ for
all $k=1,\dots, n$ there exists an integer $t_k\in \Bbb{N}$ such
that $I^{t_{k}}c_kx=0$. Set $s:=\max_{_{1\leq i\leq n}}t_i$. It is
easy to see that $I^{s+t}x=0$, which contradicts the assumption that
$x$ does not belong to $\Gamma_{_{I}}(E)$. Therefore, $\p=0:_Rx$ is
a prime ideal of $R$. Since $E_1$ is an essential extension of
$\Gamma_{_{I}}(E)$, then there exists $r\in R$ such that $0\neq
rx\in \Gamma_{_{I}}(E)$. As $r$ does not belong to $\p$ and $\p$ is
a prime ideal, then $I\subseteq \p$ which contradicts the fact that
\[
x\notin \Gamma_{_{I}}(E)=\Gamma_{_{I}}(E_1)\subseteq E_1\subseteq E,
\]
where the last containment is true by injectivity of $E$.\\
\end{proof}

In what follows the next theorem plays an important role. Before
stating it, to simplify our expressions, we give a definition.\\

\begin{defn}
Let $R$ be a ring an $k$ a nonnegative integer. We define
\[
A(R, k):=\{\p\in {\rm Spec}\,R\,\,|\,\,{\rm dim}(R/\p)=k\}, \] where
${\rm {Spec}}\,R$ denotes the prime spectrum of $R$.

\end{defn}

\begin{thm}\label{1}
Assume that $R$ be a Noetherian ring, $k$ a nonnegative integer and
let $M$ be an  $R$-module such that the intersection ${\rm
Supp}\,(M)\bigcap A(R, k)$ is an infinite set. If $\p_1,\dots,\p_n$
be prime ideals of $R$ such that
\[
\{\p_1,\dots,\p_n\}\subseteq {\rm Ass}(M)\bigcap A(R, k),
\]
then there exists a prime ideal, say $\p_{n+1}$, and a nonzero
elements of $M$,  $x$ say, satisfing the following conditions:
\begin{enumerate}
  \item  $\p_{n+1}\in {\rm Supp}\,(M)\bigcap A(R, k)\setminus
\{\p_1,\dots,\p_n\}$,
  \item  $(0:_Rx)\subseteq \p_{n+1}$,
  \item  $Rx\bigcap \Gamma_{\p_1...\p_n}(M)=0$.
\end{enumerate}
\end{thm}
\begin{proof}
Since the intersection ${\rm Supp}\,(M)\bigcap A(R, k)$ has
infinitely many elements, then it is possible to find a prime ideal
$\p_{n+1}\in A(R, k)\setminus \{ \p_1,\dots,\p_n \}$ such that
$0:_Ry\subseteq \p_{n+1}$ for some $y\in M$. By Lemma [\ref{l1}],
$y=a+b$, in which $a, b \in E_R(M)$,
\[
Rb\cap \Gamma_{\p_1...\p_n}(M)=0\,\,\, \mbox{and}\,\,\,
(\p_1...\p_n)^ta=0\,\,\,\,\,\,\,\,\,\,\,\,\,\,\,\,\,(\dag)
\]
for some natural integer $t$. Therefore,
\[
(\p_1...\p_n)^ty=(\p_1...\p_n)^tb.
\]
We claim that $0:_R(\p_1...\p_n)^tb\subseteq \p_{n+1}$. Assume the
contrary, and let $s\in R\setminus \p_{n+1}$ such that
$s(\p_1...\p_n)^ty=s(\p_1...\p_n)^tb=0$. Since $\p_{n+1}$ is  a
prime ideal then $\p_j\subseteq \p_{n+1}$, for some $j=1,\dots, n$.
In view of the fact that
\[
{\rm dim}\,R/\p_j=k={\rm dim}\, R/\p_{n+1},
\]
 we conclude that $\p_j=\p_{n+1}$
which is the desired contradiction. Therefore,
\[
0:_R(\p_1\dots\p_n)^ty\subseteq \p_{n+1}.
\]
Since $(\p_1\dots \p_n)^tb$ is a finitely generated $R$-module then
$0:_Rx\subseteq \p_{n+1}$, for some $x\in (\p_1\dots \p_n)^tb$. On
the other hand,
\[
Rx\cap \Gamma_{\p_1...\p_n}(M)\subseteq (\p_1\dots \p_n)^tb\cap
\Gamma_{\p_1...\p_n}(M)\subseteq Rb\cap \Gamma_{\p_1...\p_n}(M)=0,
\]
where the last equality is true by ($\dag$), and the proof is
completed.
\end{proof}

\begin{thm}
Assume that $R$ is a Noetherian ring, $k$ a nonnegative integer and
let $M$ be an $R$-module such that for all finitely generated
submodule $N$ of $M$ the intersection ${\rm Supp}\,(M/N)\cap A(R,
k)$ is not a finite set. Then there exists a submodule $L$ of $M$
such that the intersection ${\rm Ass}\,(M/L)\cap A(R, k)$is an
infinite set.
\end{thm}
\begin{proof}
First of all, we want to show that for each natural integer $j$
there exists a chain $T_1\subseteq T_2\subseteq \dots\subseteq T_j$,
of finitely generated submodules of $M$, prime ideals $\p_1,\dots,
\p_j \in A(R, k)$ and elements $x_1, \dots, x_j$ such that
$\p_i=T_j:_Rx_i+T_j$, for each $i=1,\dots, j$. We use induction on
$j$. Let $\p_1$ be an arbitrary element of ${\rm Supp}\,(M)\cap A(R,
k)$. There is an element $x_1\in M$ such that
$0:_Rx_1\subseteq\p_1$. Set $T_1:=\p_1x_1$. Obviously, $\p_1\in {\rm
Ass}\,(M/T_1)$ and, for $j=1$, we are done. Now, we are going to
construct $T_{j+1}$. Since, by induction hypothesis, $T_j$ is
finitely generated and the intersection ${\rm Supp}\,(M/T_j)\cap
A(R, k)$ is an infinite set then by Theorem [\ref{1}], there exists
a prime ideal
\[
\p_{j+1}\in {\rm Supp}\,(M/T_j)\cap A(R, k)\setminus \{\p_1,\dots,
\p_j\},
\] and a
nonzero element $x_{j+1}+T_j\in M/T_j$ such that
\[
(T_j:_Rx_{j+1}+T_j)\subseteq
\p_{j+1}\,\,\mbox{and}\,\,(Rx_{j+1}+T_j)/T_j\cap\Gamma_{\p_1\dots\p_j}(M/T_j)=0.
\]
Set $T_{j+1}:=\p_{j+1}x_{j+1}+T_j$. First, we prove that
$(T_{j+1}:_Rx_{i}+T_{j+1})=\p_{i}$, for $i=1,\dots, j+1$. In the
case that $i=j+1$, one just need to investigate that
$(T_{j+1}:_Rx_{j+1}+T_{j+1})$ is a subset of $\p_{j+1}$. Assume the
contrary, and let $s\in R\setminus \p_{j+1}$ be such that
$sx_{j+1}\in T_{j+1}=\p_{j+1}x_{j+1}+T_j$. Then
\[
s(Rx_{j+1}+T_{j})/T_{j}\subseteq (\p_{j+1}x_{j+1}+T_{j})
/T_{j}=\p_{j+1}(Rx_{j+1}+T_{j}) /T_{j}.
\]
This implies that
\[
s\in \sqrt{\p_{j+1}+(T_j:_RRx_{j+1})}=\p_{j+1},
\]
which is a contradiction. Therefore,
$(T_{j+1}:_Rx_{j+1}+T_{j+1})=\p_{j+1}$. On the other hand, by
induction hypothesis, for all $i=1,\dots, j$ and $s\in R\setminus
\p_i$ we have $0\neq sx_i+T_j\in \Gamma_{\p_1\dots \p_j}(M/T_j)$.
Since
\[
(Rx_{j+1}+T_j)/T_j\cap \Gamma_{\p_1\dots \p_j}(M/T_j)=0,
\]
then
\[
(Rx_{j+1}+T_j)\cap (Rsx_i+T_j) =T_j .
\] This means that $sx_i$ does not
belong to $Rx_{j+1}+T_j$, which implies that $sx_i$ does not belong
to $T_{j+1}$. Therefore,
\[
(T_{j+1}:_Rx_i+T_{j+1})=\p_i, i=1, 2,\dots, j.
\]
Set $L:=\cup_{j=1}^{\infty}T_j$. We claim that
\[
\bigcup_{j=1}^{\infty}\{\p_j\}\subseteq {\rm Ass}(M/L) \cap A(R, k).
\] It is
enough to prove that $(L:_Rx_j+L)\subseteq \p_j$. By the way of
contradiction, assume that there exits an element $s\in R\setminus
\p_j$ such that $sx_j\in L$.  Let $l$ be a natural integer such that
$sx_j\in T_l$. But, in this way $sx_j\in T_{j+l}$, which contradicts
the fact that
\[
(T_{j+l}:_Rx_j+T_{j+l})=\p_j.
\]
\end{proof}

\begin{cor}\label{c1}
Let $R$ be a Noetherian ring and $M$ a weakly Laskerian $R$-module.
Then for each nonnegative integer $k$, $M$ has a finitely generated
submodule, $N_{k}$ say, such that the intersection ${\rm
Supp}\,(M/N_k)\cap A(R, k)$ is finite.
\end{cor}
\begin{proof}
By Theorem [\ref{1}], and the definition of a weakly Laskerian
module every thing is evident.
\end{proof}

\section{\textbf{Main Result}}
\vspace{0.5cm} This section contains the main result of this paper
asserting that over a Noetherian ring the class of ${\rm FSF}$ and
weakly
Laskerian modules are the same classes.\\
\begin{defn}{\rm(See [\ref{divmafi1}], Definition 2.1)}
An $R$-module $M$ is said to be  \emph{weakly Laskerian}, if ${\rm
Ass} (M/N)$ is finite, for each proper submodule $N$ of $M$.
\end{defn}

\begin{defn}{\rm(See [\ref{hung}], Definition 2.1)}
An $R$-module $M$ is said to be $\emph{{\rm FSF}}$, if there exists
a \textbf{F}initely generated submodule $N$ of $M$, such that the
\textbf{S}upport of the quotient module  $M/N$ is a \textbf{F}inite
set.
\end{defn}

\begin{defn}
An $R$-module $M$ is said to be  $\emph{{\rm Minimax}}$, if there
exists a finitely generated submodule $N$ of $M$, such that $M/N$ is
an
Artinian.\\
\end{defn}
 The class of minimax modules was introduced by Z${\rm \ddot{o}}$chinger [\ref{zoshinger1}] and
 it was proved by Zink [\ref{zink}] and Enochs [\ref{Enochs}], that
 over a complete local ring a module is minimax if and only if it is
 Matlis reflexive.\\
\begin{rem}{\rm
Comparing the property of being Artinian, Laskerian, weakly
Laskerian or Minimax module, over a Noetherian ring, it is not hard
to see that
\begin{eqnarray*}
\rm Artinian\Rightarrow \rm Laskerian \Rightarrow \rm weakly\,\,
Laskerian,
\end{eqnarray*}
and
\begin{eqnarray*}
{\rm Artinian\Rightarrow \rm {Minimax\Rightarrow weakly \,\,\rm
Laskerian.}}
\end{eqnarray*}
One can find examples of ${\rm FSF}$ or weakly Laskerian modules
which is not Artinian, Laskerian or Minimax. However, the next
theorem shows that, over a Noetherian ring, the property of being
${\rm FSF}$ or weakly Laskerian module are the same properties.}
\end{rem}
\begin{thm}\label{3}
Let $R$ be a Noetherian ring and $M$ a nonzero $R$-module. The
following statements are equivalent:

\begin{enumerate}
  \item $M$ is a weakly Laskerian module;
  \item M is an ${\rm FSF}$ module.
\end{enumerate}
\end{thm}\label{2}
\begin{proof}
$(1)\Rightarrow (2)$ By Corollary [\ref{c1}], there exists finitely
generated submodules $N_0$ and $N_1$ such that the intersection
\[
{\rm Supp}\,(M/N_0)\cap A(R, 0) \,\,\,\,\mbox{and}\,\,\,\,{\rm
Supp}\,(M/N_1)\cap A(R, 1),
\]
are finite sets. Put $N:=N_0+N_1$.  It is enough to prove that
\[
{\rm Supp}\,(M/N)\subseteq (\,{\rm Supp}(M/N_0)\cap A(R,
0)\,)\bigcup(\,{\rm Supp}(M/N_1)\cap A(R, 1)\,).
\]
Suppose the contrary and let $\p$ be a prime ideal maximal with
respect to the property
\[
\p\in {\rm Supp}\,(M/N)\setminus [(\,{\rm Supp}(M/N_0)\cap A(R,
0))\,\bigcup\,({\rm Supp}(M/N_1)\cap A(R, 1)\,)].
\]
In view of
\[
{\rm Supp}(M/N)\subseteq {\rm Supp}(M/N_0)\cap {\rm Supp}(M/N_1),
\] we conclude that ${\rm dim}\,(R/\p)>1$ and so $V(\p)$ is not finite. For
each prime ideal $\frak Q$ which properly contains $\p$, the maximal
property of $\p$, implies that
\[
\frak Q\in [(\,{\rm Supp}(M/N_0)\cap A(R, 0)\,)\bigcup(\,{\rm
Supp}(M/N_1)\cap A(R, 1))],
\]
which contradicts the fact that the union
\[
 (\,{\rm Supp}(M/N_0)\cap
A(R, 0)\,)\bigcup(\,{\rm Supp}(M/N_1)\cap A(R, 1)),
\]
is a finite set.\\
$(2)\Rightarrow (1)$ There exists a finitely generated submodule $N$
of $M$ such that ${\rm Supp}(M/N)$ is a finite set. Let $L$ be an
arbitrary submodule of $M$.  From the short exact sequence
\[
0\rightarrow N/{N\cap L}\rightarrow M/L \rightarrow
M/{L+N}\rightarrow 0,
\]
we conclude that
\[
{\rm Ass}\,(M/L)\subseteq {\rm Ass}\,(N/{N\cap L})\,\bigcup\, {\rm
Ass}\,(M/{N+L}).
\]
This means that ${\rm Ass}\,(M/L)$ is a finite set and we are done.
\end{proof}

\begin{cor}
Let $(R, \frak m)$ be a Noetherian local ring and $M$ a weakly
Laskerian $R$-module. Then $M\otimes_RR^{*}$ is a weakly Laskerian
$R^{*}$-module, where $R^{*}$ denotes the $\frak m$-adic completion
of $R$.
\end{cor}
\begin{proof}
By Theorem [\ref{3}], there exists a finitely generated submodule
$N$ of $M$ such that ${\rm {\rm Supp}}_{_{R}}\,(M/N)$ is finite and
so ${\rm dim}_{_{R}}\,(M/N)\leq 1$. Set $J:=\bigcap_{\p\in {\rm {\rm
Supp}}\,(M/N)}\p$. Then ${\rm dim}\, (R/J)\leq 1$ and ${\rm {\rm
Supp}}_R\,(M/N)=V(J)$. It is easy to see that, $${\rm
Supp}_{_{R^{*}}}\,(M/N\otimes_RR^{*})\subseteq V(JR^{*}).$$ Since
${\rm dim}\,R^{*}/JR^{*}\leq 1$ and $R^{*}$ is a local ring then
$V(JR^{*})$ and, consequently, ${\rm {\rm {\rm
Supp}}}_{_{R^{*}}}\,(M/N\otimes_RR^{*})$ is a finite set. Now, by
Theorem [\ref{3}], we conclude that $M\otimes_RR^{*}$ is a weakly
Laskerian $R^{*}$-module.\\
\end{proof}
The following theorem shows that the property of being weakly
Laskerian ascends under the local finite integral ring extensions.\\

\begin{cor}
Let $(R, \frak m) \rightarrow (S, \frak n)$ be a Noetherian local
ring homomorphism in which $S$ is a finite $R$-module. Then every
weakly Laskerian $S$-module is a weakly Laskerian $R$-module.
\end{cor}
\begin{proof}
Let $M$ be a weakly Laskerian $S$-module. By Theorem [\ref{3}],
there exists a finitely generated submodule of $M$, say $N$, such
that ${\rm Supp}_S(M/N)$ is a finite set. Set $J:=\cap_{\p\in {\rm
Supp}_S(M/N)}\p$. Then ${\rm Supp}_S(M/N)=V(J)$ and ${\rm
dim}_S(M/N)={\rm dim}\,(S/J)\leq 1$. Set $I=J\cap R$.  Since $S/J$
is an integral extension of $R/I$ then, by lying over theorem, one
can conclude that ${\rm Supp}_R(M/N)\subseteq V(I)$.  As ${\rm
dim}\,(R/I)\leq 1$ and $R$ is a local ring then $V(I)$ and so ${\rm
Supp}_R(M/N)$ is a finite set. Now, the result is evident by Theorem
[\ref{3}].
\end{proof}

\nocite{*}
\bibliographystyle{amsplain}

\end{document}